\documentclass{birkjour} 

 \usepackage{amsopn}
 \usepackage{amsmath,amsthm,amssymb}
\usepackage{prettyref}
\usepackage{multicol}

\newcommand{\nc}{\newcommand}

 \nc{\aff}{\mathfrak{aff} } \nc{\bb}{\mathfrak{b} }
\nc{\cc}{\mathfrak{c} }  \nc{\dd}{\mathfrak{d} }
 \nc{\ggo}{\mathfrak{g} }
 \nc{\hh}{\mathfrak{h} }  \nc{\ii}{\mathfrak{i} }
 \nc{\jj}{\mathfrak{j} }  \nc{\kk}{\mathfrak{k} }
\nc{\mm}{\mathfrak{m} }   \nc{\nn}{\mathfrak{n} }
\nc{\pp}{\mathfrak{p} }  \nc{\rr}{\mathfrak{r} } \nc{\sg}{\mathfrak{s} }
 \nc{\sog}{\mathfrak{so} }  \nc{\spg}{\mathfrak{sp} }
 \nc{\sug}{\mathfrak{su} }  \nc{\slg}{\mathfrak{sl} }
 \nc{\tg}{\mathfrak{t} }  \nc{\uu}{\mathfrak{u} }
 \nc{\vv}{\mathfrak{v} } \nc{\ww}{\mathfrak{w} }
 \nc{\zz}{\mathfrak{z} }

 \nc{\ggob}{\overline{\mathfrak{g}}}

\nc{\glg}{\mathfrak{gl} }

\nc{\pca}{\mathcal{P}} \nc{\nca}{\mathcal{N}}

 \nc{\vp}{\varphi} \nc{\ddt}{\frac{{\rm d}}{{\rm d}t}}
 \nc{\la}{\langle} \nc{\ra}{\rangle}

 \nc{\SO}{{\sf SO}} \nc{\Spe}{{\sf Sp}} \nc{\Sl}{{\sf Sl}}
 \nc{\SU}{{\sf SU}} \nc{\Or}{{\sf O}} \nc{\U}{{\sf U}}
 \nc{\GL}{{\sf GL}} \nc{\Se}{{\sf S}} \nc{\Cl}{{\sf Cl}}
 \nc{\Spin}{{\sf Spin}} \nc{\Pin}{{\sf Pin}}

 \nc{\RR}{{\mathbb R}} \nc{\HH}{{\mathbb H}} \nc{\CC}{{\mathbb C}}
 \nc{\ZZ}{{\mathbb Z}} \nc{\FF}{{\mathbb F}} \nc{\NN}{{\mathbb N}}
 \nc{\GG}{{\mathbb G}} \nc{\JJ}{{\mathbb J}} \nc{\II}{{\mathbb I}}
 \nc{\KK}{{\mathbb K}} \nc{\DD}{{\mathbb D}}

 \nc{\ad}{\operatorname{ad}} \nc{\Ad}{\operatorname{Ad}}
 \nc{\coad}{\operatorname{coad}} \nc{\ct}{\operatorname{T}}
 \nc{\rank}{\operatorname{rank}} \nc{\Irr}{\operatorname{Irr}}
 \nc{\End}{\operatorname{End}} \nc{\Aut}{\operatorname{Aut}}
 \nc{\Inn}{\operatorname{Inn}} \nc{\Der}{\operatorname{Der}}
 \nc{\Dera}{\operatorname{Dera}} \nc{\Auto}{\operatorname{Auto}}
  \nc{\Iso}{\operatorname{Iso}}
 \nc{\SL}{\operatorname{SL}}

 \nc{\tr}{\operatorname{tr}}
 \nc{\rk}{\operatorname{rk}}
 \theoremstyle{plain}
 \newtheorem{teo}{Theorem}[section]
 \newtheorem{pro}[teo]{Proposition}
 \newtheorem{cor}[teo]{Corollary}
 \newtheorem{lm}[teo]{Lemma}
 \newtheorem{conj}[teo]{Conjecture}

 \renewenvironment{proof}{\noindent \emph{Proof.}}{\hfill $\blacksquare$}

 \theoremstyle{remark}
 
 \newtheorem{rem}{Remark}

 \newcommand{\R}{\mathbb R}

\newcommand{\Heis}{\sf H}
\newcommand{\so}{\mathfrak{so} }

\newcommand{\mg}{\mathfrak g }

\newcommand{\mn}{\mathfrak n }
\newcommand{\mz}{\mathfrak z }
\newcommand{\mv}{\mathfrak v }
\newcommand{\mh}{\mathfrak h }

\newcommand{\mgg}{\mathfrak g }

\renewcommand{\sl}{\mathfrak{sl} }
\newcommand{\gl}{\mathfrak{gl} }

\newcommand{\alphadot}{\stackrel{\cdot}{\alpha}}

\newcommand{\bil}{\left\langle \;,\;\right\rangle}
\newcommand{\lela}{\left \langle}
\newcommand{\rira}{\right \rangle}
\newcommand{\lra}{\longrightarrow}

\newcommand{\bs}{\backslash}

\begin{document}
\title[Homogeneous geodesics in pseudo-Riemannian nilmanifolds]{Homogeneous geodesics in pseudo-Riemannian nilmanifolds}

\author{Viviana del Barco}
\email{delbarc@fceia.unr.edu.ar}

\address{Universidad Nacional de Rosario, ECEN-FCEIA, Depto. de Matem\'a\-tica, Av. Pellegrini 250, 2000 Rosario, Santa Fe, Argentina.}

 \date{\today}

\begin{abstract} We study the geodesic orbit property for nilpotent Lie groups $N$ when endowed with a pseudo-Riemannian left-invariant metric. We consider this property with respect to different groups acting by isometries. When $N$ acts on itself by left-translations we show that it is a geodesic orbit space if and only if the metric is bi-invariant. Assuming $N$ is 2-step nilpotent and with non-degenerate center we give algebraic conditions on the Lie algebra $\mn$ of $N$ in order to verify that every geodesic is the orbit of a one-parameter subgroup of  $N\rtimes\Auto(N)$. In addition we present an example of an almost g.o. space such that for null homogeneous geodesics, the natural parameter of the orbit is not always the affine parameter of the geodesic.
\end{abstract}

\thanks{The author was partially supported by SECyT-UNR, CONICET and FONCyT. \\
{\it (2010) Mathematics Subject Classification}: 53C50 
53C22, 
53C30, 
22E25.\\
{\it Keywords: pseudo-Riemannian homogeneous spaces, homogeneous geodesics, isometric actions, nilpotent Lie groups.}
}


\maketitle

\section{Introduction}

 Homogeneous geodesics in pseudo-Riemannian spaces became of interest because of Penrose li\-mits in homogeneous spacetimes \cite{FI-ME-PH,PH}. Penrose limits preserve homogeneity if all null geodesics in the spacetime are homogeneous. This fact motivated the further study of pseudo-Riemannian homogeneous spaces for which all null geodesics are homogeneous (g.o. spaces), or almost all geodesics are so (almost g.o.). 

Riemannian g.o spaces were thoroughly investigated in \cite{KO-VA} based on Szenthe's idea of geodesic graph \cite{SZ}. Kowalski and Vanhecke classify geodesic orbit spaces up to dimension 6 and prove that this is the first dimension where non-naturally reductive g.o. spaces appear.  Gordon presents a 7-dimensional example on a Riemannian nilmanifold \cite{GO}. These dimensions are especially interesting in differential geometry.

In the pseudo-Riemannian case, 3-dimensional Lorentzian g.o. spaces are classified in \cite{CA-MA}. Homogeneous Lorentzian g.o. metrics on the oscillator group of dimension 4 are given in \cite{BA-GA-OU}. Additional examples of pseudo-Riemannian g.o. spaces can be found in \cite{DU-KO3}. The main difference between the Riemannian and pseudo-Riemannian cases is that homogeneous pseudo-Riemannian spaces do not always admit reductive decompositions \cite{FI-ME-PH}. In addition, the parameter of the one-parameter group whose orbit is a null geodesic and the affine parameter of the latter, need not agree \cite{DU-KO,DU}.

Nilpotent Lie groups and their compact quotients have been a rich source of examples in differential geometry: Thurston's symplectic non-K\"{a}hler manifold \cite{TH} and  Kaplan's examples of g.o. spaces which are not naturally reductive \cite{KA2}, are nilmanifolds. In the present work we study the geodesic orbit property for left-invariant metrics on pseudo-Riemannian nilpotent Lie groups. Relevant results on Riemannian nilmanifolds can be found in \cite{GO,LA}; for instance such a space is g.o. only when the Lie algebra is 2-step. However, this is not true in the pseudo-Riemannian case (see \cite[Example 5.7]{CH-WO}).

We characterize the homogeneous geodesics of a pseudo-Riemannian nilpotent Lie group with respect to the action of two different groups on $N$, namely $N$ itself and $G=N\rtimes \Auto(N)$, where $\Auto(N)$ is the group of isometric automorphisms. We prove that in the first case the g.o. spaces are obtained exactly when the metric is bi-invariant.

In the second case we generalize to the pseudo-Riemannian setting a result given by Gordon \cite{GO} for 2-step Riemannian g.o. nilmanifolds (see Theorem \ref{teo:teo1} below). Our more general result reveals the necessity of the reparametrization of null geodesics when the metric tensor is not positive definite.

To conclude this work we consider the following conjecture posed by Du{\v{s}}ek in \cite{DU2}:

\noindent {\bf Conjecture 1:} If an homogeneous pseudo-Riemannian space is a geodesic orbit space or an almost geodesic orbit space, then the parameter of the one-parameter group whose orbit is a null geodesic is an affine parameter for the geodesic.

We prove that this conjecture does not hold for almost g.o. homogeneous spaces by giving as counterexample a pseudo-Riemannian 2-step nilpotent Lie group.

\medskip

\noindent {\em Acknowledgments.} I am grateful to Gabriela Ovando and Aroldo Kaplan for their useful comments on a previous version of the paper. I also want to thank Gabriela Ovando for several productive discussions on the subject.

Special thanks to the referee whose suggestions helped to improve the results presented here.
\section{Preliminaries on homogeneous geodesics}

A pseudo-Riemannian manifold $M$ is homogeneous if it admits a transitive action by a connected subgroup of isometries $G$. Fixing a point $o\in M$ it is possible to identify $M$ with the homogeneous space $G/H$ where $H$ is the isotropy subgroup of $o$. Moreover, $M$ is isometric to the pseudo-Riemannian space $G/H$ with a $G$-invariant metric.

Homogeneous Riemannian spaces $(G/H,g)$ are always reductive since the Lie algebra $\mgg$ of $G$ admits an $\Ad(H)$ invariant complement $\mm$ of $\mh$, the Lie algebra of $H$. However, pseudo-Riemannian homogeneous spaces might not be reductive \cite{FI-ME-PH}. Throughout this work we will assume that the pseudo-Riemannian  space $(G/H,g)$ is reductive.

A geodesic $\gamma:J\lra G/H$ ($J$ an open interval of $\R$) through the point $o$ is an {\em homogeneous geodesic} if it is an orbit of a one-parameter group of isometries in $G$. That is, if it can be reparametrized as $\exp tX\cdot o$ for some $X\in\mgg$, where $\mgg$ is the Lie algebra of $G$. When every geodesic on $G/H$ is homogeneous the manifold $G/H$ is said to be a pseudo-Riemannian geodesic orbit space, or just a g.o. space \cite{DU}. Naturally reductive homogeneous spaces are particular cases of g.o. space.

If an homogeneous geodesic can be reparametrized as $\exp tX\cdot o$, then $X\in\mgg$ is called a {\em geodesic vector}. Geodesic vectors in pseudo-Riemannian reductive homogeneous spaces are characterized by an algebraic condition which is known as the geodesic lemma. This characterization was proved in the Riemannian case in \cite{KO-VA} and in \cite{DU-KO2} in the pseudo-Riemannian case. Nevertheless it has been repeatedly used before its formal proof; see \cite{FI-ME-PH,PH}.

\begin{lm} [Geodesic Lemma]
Let $M=G/H$ be an homogeneous manifold with reductive presentation $\mgg=\mm\oplus\mh$. An element $X\in\mgg$ is a geodesic vector if and only if there exists some constant $k\in\R$ such that
\begin{equation}\label{eq:geolema}
\lela [X,Z]_{\mm},X_{\mm}\rira=k\,\lela X_{\mm},Z\rira \mbox{ for all Z}\in\mm.
\end{equation}
\end{lm}

The subindex $\mm$ denotes the component on that subspace of the vector. Notice that according to this formula, if $X$ is a geodesic vector with constant $k$, then $\eta X$ is a geodesic vector with constant $\eta k$ for any $\eta\neq 0$.

Below we sketch the ideas of the proof of the lemma; we refer to \cite{DU-KO2,KO-VA} for further details. Let  $\alpha(t)$ be the one-parameter subgroup of a vector $X\in\mgg$, that is $\alpha(t)=\exp tX\cdot o$, $t\in \R$. By definition $d\alpha/dt=X^*$ so $\nabla_{\alphadot}\alphadot=\nabla_{X^*}X^*$ where $X^*$ is the infinitesimal vector field induced by the $G$-action on $G/H$.

Let $\gamma:J\lra G/H$ be a reparametrization of $\alpha$. Thus there exists a diffeomorphism $\varphi:\R\lra J$ such that $\gamma(\varphi(t))=\alpha(t)$. Canonical computations show that
$$
\nabla_{\stackrel{\cdot}{\gamma}} \stackrel{\cdot}{\gamma}\;=\; (\psi')^2 \;\nabla_{\stackrel{\cdot}{\alpha}} \stackrel{\cdot}{\alpha} + \;\psi''\, \stackrel{\cdot}{\alpha} \;=\;(\psi')^2  \;\nabla_{X^*} X^* +\;\psi''\, X^*,
$$
where $\psi(s)=\varphi^{-1}(s)=t$ has nowhere vanishing derivative. According to the definition above, $X\in\mgg$ is a geodesic vector if and only if for some $\varphi$ the curve $\gamma$ is a geodesic. Equivalently $X$ is a geodesic vector if and only if
\begin{equation}\label{eq:A} \nabla_{X^*} X^* =- k(s)\, X^*,\qquad \mbox{where}\quad k= \psi''/(\psi')^2.
\end{equation} Since $G$ acts by isometries on $G/H$, $k$ is a constant function.
Moreover, the Koszul formula for the metric connection of the $G$-invariant metric on $G/H$ shows that for any vector $X$ in $\mgg$ one has
$$\lela \nabla_{X^*}X^*,Z^*\rira=-\lela X^*,[X,Z]^*\rira\quad \mbox{ for all } Z\in\mgg.$$ Recall that the projection $\pi:G\lra G/H$ induces the surjective map \linebreak $d\pi_e:\mgg\lra T_oG/H$, $X\in\mgg\mapsto d\pi_e(X)=d\pi_e(X_\mm)=X^*_o$. Therefore $X\in\mgg$ is geodesic, i.e. it satisfies \prettyref{eq:A}, if and only if it verifies the condition on the geodesic lemma.

Notice that whenever $X$ is a geodesic vector, Equation \prettyref{eq:A} holds and gives the possible repara\-metrizations for the one-parameter subgroup $\exp tX\cdot o$. In fact, if this equation holds with $k=0$ then $\psi''=0$ implies $t=\psi(s)=\tilde{a}s+\tilde{b}$, so $\varphi(t)=at+b$, for some $a\neq 0$ and $t$ is the affine parameter of the geodesic $\gamma$. Also in this case, $\nabla_{\stackrel{\cdot}{\alpha}}\stackrel{\cdot}{\alpha}=0$. To the contrary, if $k\neq 0$ in Equation \prettyref{eq:A}  then $\psi$ verifies the differential equation $\psi''+k\,(\psi')^2=0$ thus $\psi(s)=\frac1k \ln(a s+b)$ for some $a\neq 0$. Hence $s=e^{-kt}$ is an affine parameter for the geodesic $\gamma$. Consequently, the affine parameters for a geodesic which is a reparametrization of $\exp\,tX\cdot o$ are $t$ and $e^{-kt}$ where $k$ arises from \prettyref{eq:A}.

It is important to remark that only null homogeneous geodesics admit $k\neq 0$. In fact, the geodesic $\gamma$ with $\gamma(\varphi(t))=\alpha(t)=\exp t X\cdot o$ is non-null only when $\lela X_\mm,X_\mm\rira\neq 0$. On the one hand, the geodesic lemma asserts that $$ k\lela X_\mm,X_\mm\rira= \lela X_\mm,[X,X_\mm]\rira;$$ and on the other hand the $\Ad(H)$-invariance of the metric in $\mm$ shows
$$\lela [X,X_\mm]_{\mm},X_{\mm}\rira =\lela[X_\mh,X_\mm],X_\mm\rira =-\lela X_\mm,[X_\mh,X_\mm]\rira.$$
Thus $\lela [X,X_\mm]_{\mm},X_{\mm}\rira=0$ which implies $k=0$. 

\medskip

The following proposition characterizes the geodesic vectors and its proof is a direct consequence of the geodesic lemma (see \cite{KO-VA} for instance).
\begin{cor}\label{pro:pro1}
Every geodesic on an homogeneous pseudo-Riemannian homogeneous space $G/H$ with reductive decomposition $\mgg=\mm+\mh$ is homogeneous if and only if each $Y\in\mm$ is a projection of a geodesic vector $X\in\mg$. That is, for each $Y\in\mm$ there exists $A\in\mh$ and $k\in\R$ such that
\begin{equation}
\lela [A+Y,Z]_{\mm},Y\rira=k\,\lela Y,Z\rira \mbox{ for all }Z\in\mm.
\end{equation}
\end{cor}

The elements in $\mm$ can be the projection of more than one geodesic vector. 

\begin{pro} \label{pro9} Let $X\in \mgg\bs\{0\}$ be a geodesic vector and $A\in\mh$. The vector $A+X$ is geodesic if and only if for some $\lambda\in\R$ one has $$[A,X_\mm]=\lambda X_\mm.$$ If $\lela X_\mm,X_\mm\rira\neq 0$ then $\lambda=0$.
\end{pro}
\begin{proof}
Let $k\in\R$ be such that 
$\lela [X,Z]_\mm,X_\mm\rira=k\lela X_\mm,Z\rira$ for all $Z\in\mm$. Then one has
$$\lela [A+X,Z]_\mm,X_\mm\rira= \lela [A,Z],X_\mm\rira+k\lela X_\mm,Z\rira=\lela Z,kX_\mm-[A,X_\mm]\rira,$$ where for the last equality one uses the fact that $\mh$ acts by skew-symmetric transformations on $\mm$.

If $A+X$ is a geodesic vector, there exists $k'\in\R$ such that $$ \lela Z,kX_\mm-[A,X_\mm]\rira=k'\lela X_\mm,Z\rira \mbox{ for all } Z\in\mm$$ because of the geodesic lemma. This, together with the fact that the metric in $\mm$ is non-degenerate, implies $[A,X_\mm]=(k-k')X_\mm$, so the equation above holds.

Conversely, if $[A,X_\mm]=\lambda X_\mm$ for some $\lambda\in\R$, then
$$\lela [A+X,Z]_\mm,X_\mm\rira= \lela Z,kX_\mm-\lambda X_\mm\rira=(k-\lambda)\lela Z,(A+X)_\mm\rira$$ and $A+X$ is geodesic.

The last assertion follows from the fact that $\lela [A,X_\mm],X_\mm\rira=0 $ since $A\in\mh$ is skew-symmetric with respect to $\bil$.
\end{proof}

\medskip
One of the techniques used to study pseudo-Riemannian g.o. spaces is the concept of {\em geodesic graph} which we introduce below. The original idea comes from Szenthe \cite{SZ}.

Let $(G/H,g)$ be a pseudo-Riemannian reductive homogeneous space and $\mg=\mm\oplus \mh$ an $\Ad(H)$-invariant decomposition of the Lie algebra $\mg$. A map $\xi:\mm\lra \mh$ is $\Ad(H)$-equivariant if $\xi(\Ad(h)X)=\Ad(h) \xi(X)$ for all $X\in\mm$ and $h\in H$. This map is said to be {\em rational} if the components $\xi_i$ of $\xi$, with respect to a basis of $\mm$ and a basis of $\mh$, are rational functions of the form $\xi_i=P_i/P$, where $P_i$ and $P$ are homogeneous polynomials  (on coordinates on $\mm$) and $\deg P_i=\deg P+1$.

A {\em geodesic graph} of the homogeneous space $G/H$ is an  $\Ad(H)$-equivariant map $\xi:\mm\lra \mh$ which is rational on an open dense subset $\mathcal U$ of $\mm$ and such that $X+\xi(X)$ is a geodesic vector for each $X\in\mm$. A pseudo-Riemannian manifold is naturally reductive if it admits a linear geodesic graph \cite{DU}.

On every reductive g.o. space there exists at least one geodesic graph \cite{SZ}. A reductive homogeneous space $G/H$ with reductive decomposition $\mgg=\mm\oplus \mh$ is an {\em almost g.o. space} (resp. {\em n.g.o}) space if a geodesic graph can be defined on an open dense subset $\mathcal U\subseteq \mm$ (resp. on the null cone $\mathcal N\subseteq \mm$), but not on all $\mm$. Here n.g.o. means null geodesics are orbits; this notion was introduced by Meessen \cite{ME}.

\section{Homogeneous geodesics on nilpotent Lie groups}

The purpose of this section is to study homogeneous geodesics on nilpotent Lie groups when endowed with a left-invariant pseudo-Riemannian metric. Throughout the section, $(N, \bil)$ denotes a connected and simply connected nilpotent Lie group $N$ equipped with a left-invariant pseudo-Riemannian metric $\bil$; we also denote with $\bil$ the metric on the Lie algebra $\mn$ of $N$. 

We determine algebraic conditions for $N$ to be a geodesic orbit space with respect to the particular reductive presentation $N=G/H$ where $G=N\rtimes \Auto(N)$ and $H=\Auto(N)$. Recall that $\Auto(N)$ is the subgroup of isometric automorphisms of $N$. Since $N$ is simply connected, we identify $\Auto(N)$ with the group of isometric automorphisms of $\mn$, $\Auto(\mn)$. 

Denote with $\Dera(\mn)$ the Lie algebra of skew-symmetric derivations of $\mn$; this is the Lie algebra of $\Auto(\mn)$. Then $\mgg=\mn\oplus\Dera(\mn)$ is a reductive decomposition of $N=G/\Auto(\mn)$, although $\mgg$ might admit different reductive decompositions.

The following result shows that the property of $N$ of being a geodesic orbit space with respect to the action of $G=N\rtimes \Auto(\mn)$ does not depend on the reductive decomposition of $\mgg$ considered (see \cite[Theorem 3.5.]{LA3} for Riemannian na\-turally reductive 2-step nilpotent Lie groups).

\begin{teo} \label{teo:teo31}Let $N$ be a pseudo-Riemannian simply connected nilpotent Lie group and let $G=N\rtimes \Auto(\mn)$ act on $N$ by isometries. The geodesics of $N$ are homogeneous with respect to $G=N\rtimes \Auto(\mn)$ if and only if they are homogeneous with respect to the reductive decomposition $\mgg=\mn\oplus \Dera(\mn)$.
\end{teo}

\begin{proof}
Let $\mgg=\mm\oplus \Dera(\mn)$ be a reductive decomposition of the homogeneous space $N=G/\Auto(\mn)$. Let $P:\mgg\lra\Dera(\mn)$  and $Q:\mgg\lra\mm$ be the projections with respect to the decomposition $\mgg=\mm\oplus\Dera(\mn)$ and let $\rho:\mn\lra \Dera(\mn)$ be the map defined as $\rho(Y)=-P(Y)$. Then \cite{GO2}
$$\mm=\{Y+ \rho(Y):\;Y\in\mn\}$$
and $\psi:\mn\lra \mm$, given by $\psi(Y)=Y+\rho(Y)$ is an isometry from $(\mm,\bil_\mm)$ to $(\mn,\bil)$. Notice that $\psi$ is the restriction to $\mn$ of the projection $Q$. Moreover, the $\mm$-component of $[\psi(Y),\psi(Y')]$ satisfies
\begin{eqnarray}
[\psi(Y),\psi(Y')]_\mm&=&\psi([Y,Y'])+\psi(\rho(Y)Y')-\psi(\rho(Y')(Y)\nonumber\\
&=&\left\{[Y,Y']+\rho(Y)Y'-\rho(Y')(Y)\right\}_{\mm}.\nonumber
\end{eqnarray}
All these elements together imply that for any $A\in \Dera(\mn)$ and $U,U'\in\mm$ 
\begin{equation}\label{eq:eq100}
\lela [A+U,U']_{\mm},U\rira_{\mm}=\lela [A+\rho(Y)+Y,Y'],Y\rira 
\;\mbox{ and } \;\lela U,U'\rira_{\mm}=\lela Y,Y'\rira,
\end{equation}
where $Y,Y'$ are the unique elements in $\mn$ such that $U=Y+\rho(Y)$ and $U'=Y'+\rho(Y')$.

Assume now that every geodesic of $N$ is homogeneous with respect to $\mgg=\mm\oplus \Dera(\mn)$ and let $Y\in\mn$. We shall prove that $Y$ is the projection of a geodesic vector in $\mgg$.

Corollary \ref{pro:pro1} applied to $U=Y+\rho(Y)\in\mm$ implies that there exists $A\in\Dera(\mn)$ and $k\in \R$ such that
$$\lela[A+U,U']_\mm,U\rira_\mm=k\lela U,U'\rira_{\mm},\qquad \mbox{for all } U'\in\mm. $$
Then by \prettyref{eq:eq100} and taking $D=A+\rho(Y)$ one has
$$\lela[D+Y,Y'],Y\rira=k\lela Y,Y'\rira,\qquad \mbox{for all } Y'\in\mn. $$ So Corollary \ref{pro:pro1} shows that every geodesic is homogeneous with respect to the reductive decomposition $\mgg=\mn\oplus\Dera(\mn)$, as we wanted to prove.

The converse follows with an analogous argument.
\end{proof}\medskip

Now we focus on the situation where $N$ is 2-step nilpotent and the metric $\bil$ restricted to the center is nondegenerate. Our goal is to characterize among this family, those nilpotent Lie groups whose geodesics are one-parameter subgroups (possibly after reparametrization) of the group $G=N\rtimes \Auto(\mn)$. 
 \smallskip

Let $\mz$  be the center of $\mn$ and assume the restriction of $\bil$ to $\mz$ is nondegenerate. Hence  there exists an orthogonal decomposition of $\mn$
\begin{equation}
\label{eq:ortdec} \mn=\mv\oplus \mz
\end{equation}
so that  $\mv$ is also nondegenerate. Namely, $\mv=\{X\in\mn:\,\lela X,Z\rira=0$ for all $Z\in\mz\}$.
Each $Z\in\mz$  defines a linear transformation
$j(Z):\mv\lra\mv$ such that
\begin{equation}\label{eq:jota}
\lela j(Z)X,X'\rira =\lela Z,[X,X']\rira \quad \mbox{ for all } X,X'\in\mv.
\end{equation}
The map $j(Z)$ belongs to $\so(\mv)$, the Lie algebra of skew-symmetric maps of $\mv$ with respect to $\bil$, and also $j: \mz \lra \so(\mv)$ is a linear homomorphism. As in the Riemannian case, the maps $j(Z)$ capture important geometric information of the pseudo-Riemannian space $(N,\bil)$. 

Given $h\in \Auto(\mn)$ it is known that $\Ad(h)D=hDh^{-1}$ for each skew-symmetric derivation $D$ of $\mn$. From Equation \prettyref{eq:jota} and usual computations it follows that
\begin{equation}\label{eq:pro5}
h(j(Z)(X))= j(hZ)(hX) \mbox{ for all } X\in\mv\mbox{ and } Z\in\mz. \end{equation}
Let $Y\in\mn$ and write $Y=X+Z$ with $X\in\mv$ and $Z\in\mz$. The next result introduces necessary and sufficient conditions for $Y$ to be the projection over $\mn$ of a geodesic vector in $\mgg$.

\begin{lm} \label{lm5} An element $Y=X+Z$ of $\mn$ is the projection of a geodesic vector if and only if there exists a skew-symmetric derivation $D$ and a constant $k\in \R$ such that $D(Z)=-kZ$ and $(D+kI)X=j(Z)X$. If this holds, for any $h\in \Auto(\mn)$, $\Ad(h)D + \Ad(h)(X+Z)$ is also a geodesic vector with the same constant $k$.
\end{lm}

\begin{proof}
Notice that $Y\in\mn$ is the projection of a geodesic vector in $\mgg$ if there exists a skew-symmetric derivation $D$ such that $D+Y$ is a geodesic vector. Let $X\in\mv$ and $Z\in\mz$ be such that $Y=X+Z$. By the geodesic lemma, $D+X+Z$ is a geodesic vector if and only if there is a real constant $k$ verifying
\begin{eqnarray}\label{eq:eq1}
\lela [D+X+Z,U]_{\mn},X+Z\rira &=& \lela D\,U,X\rira+\lela D\,U,Z\rira+\lela [X,U],Z\rira \nonumber \\
&=&k\,\lela X,U\rira+k\,\lela Z,U\rira\quad \mbox{ for all } U\in\mn.
\end{eqnarray}
This equality holds because any derivation preserves the center, which implies, that it also preserves $\mv$.

Suppose $U\in \mz$, then \prettyref{eq:eq1} is equi\-valent to
$$
\lela D\,U,Z\rira=-\lela U,D\,Z\rira=k\,\lela Z,U\rira 
$$
which holds for any $U\in\mz$ if and only if $D\,Z=-kZ$ for some $k$.

Suppose $U\in \mv$ and $D$ is a skew-symmetric derivation, then \prettyref{eq:eq1} is equi\-valent to
$$\lela D\,U,X\rira+\lela [X,U],Z\rira= \lela U,-D\,X\rira+\lela j(Z)\,X,U\rira =k\,\lela X,U\rira
$$
which holds for any $U\in\mv$ if and only if $(D+kI)X=j(Z)X$ for some $k$.

Let $h\in\Auto(\mn)$. Then $\Ad(h)$ preserves the orthogonal splitting $\mn=\mv\oplus \mz$. The equality in \prettyref{eq:pro5} yields \begin{eqnarray}
\left(\Ad(h)D+kI\right)(\Ad(h)X)&=&\left(hDh^{-1}+kI\right)(hX)=hD(X)+khX\nonumber\\
&=&h(j(Z)X)=j(hZ)(hX)\nonumber\\
&=&j(\Ad(h)Z)(\Ad(h)X)\nonumber \end{eqnarray}
and also,
\begin{eqnarray}
\Ad(h)D(\Ad(h)Z)&=&hDh^{-1}(hZ)\\
&=& hDZ\;=\;-k\,hZ\;=\;-k\,\Ad(h)Z,\nonumber
\end{eqnarray}
so the result follows.
\end{proof}

\smallskip

Recall that the geodesic orbit property on nilpotent Lie groups with respect to $G=N\rtimes \Auto(\mn)$ does not depend on the reductive decomposition considered as shown in Theorem \ref{teo:teo31}. Now we can establish the following. 

\begin{teo} \label{teo:teo1} Let $N$ be a pseudo-Riemannian simply connected 2-step nilpotent Lie group $N$ having nondegenerate center. $N$ is g.o. with respect to $G=N\rtimes \Auto(\mn)$ if and only if for each $X\in\mv$ and $Z\in \mz$, there exists a $k\in \R$ and a skew symmetric derivation $D$ of $\mn$ such that $D(Z)=-kZ$ and $(D+kI)X=j(Z)X$.
\end{teo}

\begin{rem}
Recall that if the metric $\bil$ is positive definite then $k=0$ for every geodesic vector. In this case, the statement of the previous theorem coincides with that of Theorem 2.10 in \cite{GO}.
\end{rem}

Given $X+Z\in\mn$ it is possible to have more than one skew-symmetric derivation satisfying the conditions in Theorem \ref{teo:teo1} (see Proposition \ref{pro9}). If there is an open dense subset $\mathcal U$ in $\mn$, such that each $Y=X+Z\in\mathcal U$ has a unique solution to those conditions, then 
a geodesic map is defined as follows: for $Y\in\mn$, take  $\xi(Y)=D\in \Dera(\mn)$, where $D$ is the {\em unique} derivation in the previous theorem. The map $\xi:\mn\lra \Dera(\mn)$ is $\Ad(H)$-invariant as a consequence of the previous lemma and is rational because the solutions are obtained by Cramer's rule. $N$ is g.o. when $\mathcal U$ coincides with $\mn$.

\subsection{Trivial isotropy}\label{trivial}

This final part of the section is devoted to investigating homogeneous geodesics on nilpotent Lie groups $N$ when the group acting by isometries is $G=N$ and the action is by left-translations. Notice that the isotropy is $H=\{e\}$ since this action of $G$ on $N$ is free.

We obtain that a left-invariant metric on a nilpotent Lie group $N$ is g.o. with respect to the trivial isotropy presentation only when it is bi-invariant. This result is valid for any nilpotent Lie group, not only for 2-step. When we restrict to the family of 2-step nilpotent Lie groups with nondegenerate center we show that they are never g.o nor almost g.o spaces when the trivial isotropy presentation is considered.
\smallskip

Recall that a {\em bi-invariant pseudo-Riemannian metric} $\bil$ on a Lie group $N$ is a metric for which translations on the right and on the left by elements of $N$ are isometries. Let $N$ be a connected Lie group endowed with a left-invariant metric and let $\mn$ be its corresponding Lie algebra. Then the following statements are equivalent:
\begin{enumerate}
\item $\la\, ,\,\ra$  is right-invariant, hence bi-invariant;
\item\label{adinv} $\la \ad_X Y, Z\ra + \la Y, \ad_X Z \ra= 0$ for all $X,Y,Z \in \mn$;
\item \label{adinv2}$\lela \ad_X Z,X\rira=0$ for all $X,Z\in\mn;$
\item  the geodesics of $N$ starting at the identity element $e$ are the one-parameter subgroups of $N$, $\exp tX$, where $t$ is the affine parameter for the geodesic.
\end{enumerate}

The proof of the equivalences between $1.,2.$ and $4.$ can be found in \cite[Ch. 11]{ON}. It is immediate that condition $2.$ implies $3.$; for the converse, polarize $3.$ to obtain $2.$. A symmetric bilinear form on a Lie algebra $\mn$ satisfying \ref{adinv}. above is said to be {\em ad-invariant}. These equalities are valid for any connected Lie group, not only for nilpotent ones.

\smallskip

We say that a nilpotent Lie group $N$ is a {\em g.o. Lie group} if $N$ is a geodesic orbit space with respect to the action of $G=N$ on itself by left translations. 
It is clear that bi-invariant metrics are naturally reductive (and therefore g.o.) with respect to the presentation $G/H$ with $G=N$ and $H=\{e\}$. The next result shows the converse of this fact. 

\begin{teo}\label{teo:teo7}
Let $N$ be a connected nilpotent Lie group endowed with a left-invariant metric $\bil$. $N$ is a g.o. Lie group if and only if its metric is bi-invariant.
\end{teo}

\begin{proof} Assume $N$ is a g.o. Lie group and let $X\in\mn$. Using the geodesic lemma, there exists some constant $k\in\R$ such that
\begin{equation}\label{eq:glmlie}\lela [X,Z],X\rira = k\lela X,Z\rira\qquad \mbox{for all } Z\in\mn.\end{equation}
We prove that $k$ is an eigenvalue of $\ad_X$, the adjoint transformation defined by $X$. Indeed, Equation \prettyref{eq:glmlie} is equivalent to 
\begin{equation}
\label{eq:eq52}
\lela (\ad_X-k I)Z,X\rira=0 \qquad \mbox{ for all } Z\in\mn.\end{equation} If $k$ is not an eigenvalue of $\ad_X$ then $(\ad_X-k I)$ is in particular surjective which together with Equation \prettyref{eq:eq52} implies that the metric $\bil$ is degenerate. Hence, for each $X\in\mn$, $k$ is an eigenvalue of $\ad_X$. Recall that every adjoint operator $\ad_X$ is nilpotent if the Lie algebra is nilpotent, so $k=0$ for each $X\in\mn$. Thus for every $X\in\mn$ one has \begin{equation}\nonumber \lela \ad_X Z,X\rira=0 \qquad \mbox{ for all } Z\in\mn,\end{equation} which by \ref{adinv2}. above implies that the metric is bi-invariant.
\end{proof}

 \smallskip

Nilpotent Lie groups endowed with bi-invariant metrics have degenerate center. Indeed, this is a consequence of the equivalences mentioned above. The last result in particular shows that nilpotent Lie groups with non-degenerate center are never geodesic orbit Lie groups.
Nevertheless, geodesic vectors with respect to $G=N$ might exist.

Suppose that the center is non-degenerate and recall the orthogonal decomposition of $\mn$, $\mn=\mv\oplus \mz$ as in \prettyref{eq:ortdec} and the map $j:\mv\lra \mathfrak{so}(\mv)$ defined as in \prettyref{eq:jota}. Given $Y\in\mn$, denote by $X$ and $Z$ the orthogonal projection of $Y$ on $\mv$ and $\mz$ respectively, so that $Y=X+Z$. The next result characterizes geodesic vectors in $\mn$.

\begin{pro}
Let $N$ be a simply connected 2-step nilpotent Lie group with Lie algebra $\mn$ and let $\bil$ be a left-invariant pseudo-Riemannian metric on $N$ for which the center is nondegenerate. An element $Y\in\mn$ is a geodesic vector if and only if $j(Z)(X)=0$.
\end{pro}
\begin{proof}
Assume that $Y\in\mn$ is a geodesics vector,  $Y=X+Z$ with $X\in\mv$ and $Z\in\mz$. Then, as before, there exists $k\in\R$ such that
\begin{equation}\lela [Y,U],Y\rira = k\lela Y,U\rira\qquad \mbox{for all } U\in\mn.\end{equation} 
In particular, for all $U\in\mz$ one has $0=\lela [Y,U],Y\rira = k\lela Y,U\rira=k\lela Z,U\rira$; the nondegeneracy of the center implies that $k=0$ if $Y\notin\mv$. If $Y\in\mv$, then for all $U\in\mv$ we have $\lela [Y,U],Y\rira =0$ because $[Y,U]\in\mz$ and $\mz\bot\mv$. Then $0= k\lela Y,U\rira$ for all $U\in\mv$, so $k=0$ if $Y\neq 0$. We get that if $Y$ is geodesic then $\lela [Y,U],Y\rira = 0$, for all $U\in\mn$. Notice that $\lela [Y,U],Y\rira=\lela [X,U],Z\rira=\lela j(Z)(X),U\rira=0$ for all $U\in\mn$, which implies $j(Z)(X)=0$ as we intended to prove.

The converse is straightforward. 
\end{proof}

The set $\mathcal U=\{X+Z\in\mn:X\in\mv,\,Z\in\mz\mbox{ and }j(Z)(X)=0\}$ is never an open dense subset of $\mn$ unless $N$ is abelian. Thus we obtain

\begin{cor}\label{cor4}
Let $N$ be a 2-step nilpotent Lie group endowed with a pseudo-Riemannian metric such that the center is nondegenerate. Then $N$ is neither a g.o. space nor an almost g.o. space with respect to the trivial isotropy presentation.
\end{cor}

\section{An almost g.o. space whose null homogeneous geodesics require a reparametrization}

In this section we present an example of a nilpotent Lie group $N$ of dimension six which is almost g.o. when the action of $G=N\rtimes \Auto(\mn)$ on $N$ is considered. As an homogeneous manifold it is almost g.o. but not g.o. since it admits geodesics which are not homogeneous with respect to $G$. Moreover, it admits null homogeneous geodesics for which the corresponding geodesic vectors satisfy the geodesic lemma with non-zero parameter $k$. 
\smallskip

Consider $\R^6$ with the canonical differentiable structure and let $g$ denote the following pseudo-Riemannian metric on $\R^6$:
\begin{eqnarray}\label{eq:g}
g&=&\frac12 (x_3dx^1 -x_1dx^3)(x_4dx^2-x_2dx^4)+dx^5(x_4dx^2-x_2dx^4)\\
&&\hspace{1cm}+dx^6(x_3dx^1-x_1dx^3)+dx^2dx^3-dx^1dx^4+2dx^5dx^6.\nonumber
\end{eqnarray}
The pseudo-Riemannian manifold $(\R^6,g)$ admits a transitive and simple action of the 2-step nilpotent Lie group $N$ which is modeled on $\R^6$ with multiplication law such that for $p=(x_1,x_2,x_3,x_4,x_5,x_6)$ and $q=(y_1,y_2,y_3,y_4,y_5,y_6)$ one has

\begin{eqnarray}\label{eq:prodN}
p\cdot q&=&\left( x_1+y_1,x_2+y_2,x_3+y_3,x_4+y_4,\phantom{\frac12}\right.\nonumber\\
&&\;\quad\left. x_5+y_5+\frac12 (x_1y_3-x_3y_1), x_6+y_6+\frac12 (x_2y_4-x_4y_2)\right).
\end{eqnarray}

The corresponding metric on $N$ induced by \prettyref{eq:g} is invariant under left-translations. Thus $N$ is a Lie group endowed with a left-invariant pseudo-Riemannian metric. As a Lie group, $N$ is isomorphic to the product of two Heisenberg Lie groups: $N\simeq \Heis_3(R)\times \Heis_3(\R)$. Nevertheless, the metric is not a product metric. A basis of left-invariant vector fields, evaluated at a point $p=(x_1,x_2,x_3,x_4,x_5,x_6)$ is given by
$$
\begin{array}{rclcrcl}
X_1&=&\partial_1-\displaystyle{\frac{x_3}2\,\partial_5}, &\quad&X_2&=&\partial_2-\displaystyle{\frac{x_4}2\,\partial_6},\\
\vspace{-0.25cm}\\
X_3&=&\partial_3+\displaystyle{\frac{x_1}2\,\partial_5},&\quad&X_4&=&\partial_4+\displaystyle{\frac{x_2}2\,\partial_6},\\
\vspace{-0.25cm}\\
X_5&=&\partial_5, &\quad&X_6&=&\partial_6,
\end{array}
$$
where $\partial_i=\partial/\partial x_i$ is the usual coordinate system of $\R^6$. These vector fields give a basis of $\mn$, the Lie algebra of $N$, and they satisfy $[X_1,X_3]=X_5$, $[X_2,X_4]=X_6$, while the other brackets are zero. The metric on this basis of invariant vector fields is constant and at each point $p\in N$ one has
$$1=-g( X_1,X_4)=g( X_2,X_3), \quad 2=g( X_5,X_6).$$
We denote $\lela \,,\,\rira$ the metric induced on the Lie algebra $\mn$ by $g$. The center of $\mn$ is $\mz=span\{X_5,X_6\}$,
 the metric restricted to $\mz$ is nondegenerate and $\mz^\bot=\mv=span\{X_1,X_2,X_3,X_4\}$.
The linear map $j:\mz\lra \so(\mv)$ defined in \prettyref{eq:jota} evaluated on basis elements is
\begin{equation}
\label{eq:jotas}j(X_5)=-2(E_{12}+E_{34}),\quad j(X_6)=2 (E_{21}+E_{43}); \end{equation}
here $E_{ij}$ denotes the $4\times 4$ matrix which has a $1$ in the file $i$ and column $j$ and $0$ otherwise. 

Let $Z$ be a central element and $Z=z_5X_5+z_6X_6$, then $j(Z)^2= -4 z_5z_6 Id_\mv =-\lela Z,Z\rira Id_\mv$. Hence $\mn$ (resp. $N$) is a pseudo-$H$-type Lie algebra (resp. Lie group)\footnote{Riemannian $H$-type Lie algebras were introduced by Kaplan in \cite{Ka1}. In pseudo-Riemannian geometry, we find this name for the first time in \cite{CI}.}.

\begin{rem}
The nilpotent Lie group $N$ that we work with in this section is obtained as a quotient by a central element of the 7-dimensional nilpotent pseudo-$H$-type Lie group in Example 2.1 of \cite{JA-PA-PA}.
\end{rem}

\begin{pro}\label{pro4}
The pseudo-Riemannian manifold $(\R^6,g)$ is homogeneous, its isometry group has eight connected components and the connected component of the identity is isomorphic to $(\Heis_3(\R) \times \Heis_3(\R))\rtimes \GL_0(2)$. \end{pro}

\begin{proof}
As stated above, the nilpotent Lie group $N$ modeled on $\R^6$ with multiplication law as in \prettyref{eq:prodN} acts by isometries on $(\R^6,g)$ and this action is simple and transitive. So $(\R^6,g)$ is homogeneous and it is isometric to $N$ with the left-invariant metric $g$.

The full isometry group of $N$, $\Iso(N)$, (which coincides with that of $(\R^6,g)$) is isomorphic to $N\rtimes \Auto(\mn)$ where $\Auto(\mn)$ is the group of isometric automorphisms of $\mn$. In fact, since $N$ is a of pseudo-$H$-type, one may apply Theorem 1 in \cite{dBO2}. Thus the connected component of the identity of $\Iso(N)$ is $\Iso_0(N)\simeq N\rtimes \Auto_0(\mn)$.

Given $A\in\Auto(\mn)$, $A$ preserves the center since it is an automorphism of $\mn$ and, because it is an isometry, it also preserves its orthogonal complement $\mz^\bot=\mv$. Following canonical computations one obtains the elements of $\Auto(\mn)$ and shows that it has eight connected components. Below we describe the subgroup $\Auto_0(\mn)$.

For each $2\times 2$ real matrix $\tau=(\tau_{ij})_{i,j=1}^2$ such that $\det \tau\neq 0$, let $A_\tau$ be the endomorphism of $\mn$ that in the basis $\{X_i\}_{i=1}^6$ above satisfies 
$$\begin{array}{rclcrcl}
A_\tau X_1&=&\tau_{11} X_1+\tau_{21}X_3,&\quad& A_\tau X_2&=&\displaystyle{\frac{\tau_{11}}{\det \tau}X_2}+\displaystyle{\frac{\tau_{21}}{\det \tau} X_4},\\
\vspace{-0,2cm}\\
A_\tau X_3&=&\tau_{12}X_1+\tau_{22} X_3,&\quad& A_\tau X_4&=&\displaystyle{\frac{\tau_{12}}{\det \tau}X_2}+\displaystyle{\frac{\tau_{22}}{\det \tau} X_4},\\
\vspace{-0,2cm}\\
A_\tau X_5&=&\det \tau \,X_5,&\quad& A_\tau X_6&=&\displaystyle{\frac{1}{\det\tau} X_6}.
\end{array}$$

It is easy to verify that $A_\tau$ is an isometric automorphism of $\mn$. The connected component of the identity of $\Auto(\mn)$ is
\begin{equation}\label{eq:eq50} \Auto_0(\mn) = \{A_\tau:\tau\in\GL_0(2)\}\end{equation} where $\GL_0(2)$ consists of matrices of positive determinant. The product is $A_\tau\,A_\sigma=A_{\tau\cdot\sigma}$ where $\cdot$ the usual product in $\GL(2)$. So $\Auto_0(\mn)\simeq \GL_0(2)$ and $\Iso_0(N)\simeq (\Heis_3(\R)\times\Heis_3(\R))\rtimes \GL_0(2)$.

Let  $B_i:\mn\lra\mn$, $i=1,2,3$ be the endomorphisms having the following matrix representation in the basis $\{X_i\}_{i=1}^6$:
\begin{eqnarray}
B_1&=&E_{13}+E_{24}-E_{31}-E_{42}+E_{55}+E_{66},\nonumber\\
B_2&=&E_{12}+E_{14}+E_{21}+E_{23}+E_{32}+E_{41}-E_{56}-E_{65},\nonumber\\
B_3&=&E_{14}+E_{23}+E_{32}+E_{34}+E_{41}+E_{44}-E_{56}-E_{65};\nonumber
\end{eqnarray}
here $E_{ij}$ denotes the $6\times 6$ matrix which has a 1 in the file $i$ and column $j$ and $0$ otherwise. For all $i$, $B_i$ is an isometric automorphism of $\mn$. 

The other connected components of $\Auto(\mn)$ are $\Auto_i(\mn)= B_i \cdot \Auto_0(\mn)$ for $i=1,2,3$, $\Auto_4(\mn)=\{A_\tau:\tau\in\GL(2)\mbox{ and } \det\tau<0\}$ and $\Auto_{i}(\mn)=B_{i-4}\cdot\Auto_4(\mn)$ for $i=5,6,7$.
\end{proof}

\medskip

Because of Corollary \ref{cor4}, the Lie group $N$ is neither a g.o. nor an almost g.o. space with respect to the action of $N$ on itself by left-translations (hence with trivial isotropy). So we investigate the geodesic orbit property of $(N,g)$ with respect to the action of $\Iso(N)\simeq N\rtimes \Auto(\mn)$.

\begin{rem} The homogeneous manifold $(N,g)$ is not naturally reductive with respect to $\Iso(N)$. This fact is a direct consequence of Theorem 3.2 in \cite{OV}. Indeed, $[j(X_5),j(X_6)]\notin j(\mz)$, so $j(\mz)$ is not a subalgebra of $\so(\mv)$. \end{rem}

Next, we apply Lemma \ref{lm5} to determine which geodesics of $N$ are homogeneous with respect to $\Iso(N)$.


We start with the following change of basis of $\mn$: $e_i=X_i$, $i=1,\ldots,4$ and $e_5=X_5+\frac14X_6$, $e_6=-X_5+\frac14X_6$; this is a pseudo-orthonormal basis on $\mz$ with $1=\lela e_5,e_5\rira =-\lela e_6,e_6\rira $ and non-zero Lie brackets
$$[e_1,e_3]=\frac12(e_5-e_6),\quad [e_2,e_4]=2(e_5+e_6).$$
One has $\mz=span\{e_5,e_6\}$ and $\mv=span\{e_1,e_2,e_3,e_4\}$.

\begin{cor}\label{pro:pro3} The Lie algebra of skew-symmetric derivations of $\mn$ is
$$\Dera(\mn)\simeq \R\oplus \sl(2,\R)\simeq \gl(2,\R).$$
\end{cor}

\begin{proof} The structure of the Lie algebra is a consequence of \prettyref{eq:eq50}. A basis of $\Dera(\mn)$ is given by the set $\{T,H,E,F\}$, each of them being an endomorphism of $\mn$ such that in the basis $\{e_i\}_{i=1}^6$ above one has the following matrix representation:
\begin{equation}\nonumber
\begin{array}{cc}
\begin{array}{lcl}
T&=&E_{11}-E_{22}+E_{33}-E_{44}-2E_{56}-2E_{65}, \\
H&=&E_{11}+E_{22}-E_{33}-E_{44},\end{array}&
\begin{array}{lcl}
E&=&E_{13}+E_{24},\\
F&=&E_{31}+E_{42}.\end{array}\end{array}\end{equation}
Again, $E_{ij}$ denotes the $6\times 6$ matrix whose all the entries are zero except for the $ij$ entry, which is one.
The only non-vanishing Lie brackets of this basis are
$$[E,F]=H,\quad [H,E]=2E\;\;\mbox{ and }\;\; [H,F]=-2F.$$
\end{proof}

We show below that $N$ is an almost geodesic orbit space with respect to $\Iso(N)$. To do so, we define a geodesic graph $\xi:\mathcal U\lra \Dera(\mn)$ with $\mathcal U$ an open dense subset of $\mn$.

According to Lemma \ref{lm5}, an element $Y=X+Z$ in $\mn$ with $X\in\mv$ and $Z\in\mz$ is a geodesic vector if there exists a constant $k\in\R$ and a skew-symmetric derivation $D:\mn\lra \mn$ such that $(D+kI)Z=0$ and $(D+kI)X=j(Z)X$. 

Let $Y$ be an element of $\mn$. Then $Y=X+Z$ where $X=x_1 e_1+x_2e_2+x_3e_3+x_4e_4\in\mv$ and $Z=z_5e_5+z_6e_6\in\mz$. A skew-symmetric derivation has the form $\xi_1 T+\xi_2 H+\xi_3 E+\xi_4 F$, where $\xi_i$ are real numbers and $\{T,H,E,F\}$ is as in Corollary \ref{pro:pro3}. Thus the coefficients $x_i,\,z_i,\,\xi_i$ must satisfy the following system of equations.
\begin{equation}\label{eq:matriz}\begin{array}{cccc} \left(\begin{array}{cccc}
 x_{{1}}& x_{{1}}& x_{{3}}&0\\
 -x_{{2}}& x_{{2}}& x_{{4}}&0\\
x_{{3}}&- x_{{3}}&0& x_{{1}}\\
-x_{{4}}&- x_{{4}}&0& x_{{2}}\\
-2\, z_{{6}}&0&0&0\\
-2\, z_{{5}}&0&0&0
\end{array} \right) &
\left(\begin{array}{c}\xi_1\\\xi_2\\\xi_3\\\xi_4\end{array} \right)&=&
\left(\begin{array}{c}
-2(z_5-z_6)x_2-kx_1\\
\frac12(z_5+z_6)x_1-kx_2\\
-2(z_5-z_6)x_4-kx_3\\
\frac12(z_5+z_6)x_3-kx_4\\
-kz_5\\
-kz_6
\end{array}\right).
\end{array}
\end{equation}

In the case that $x_1=x_2=x_3=x_4=0$ or $z_5=z_6=0$ we choose $k=0$ and $\xi(X)=0$, which clearly solve the system in Equation \prettyref{eq:matriz}. In what follows, unless otherwise stated, it is $x_i\neq 0$ and $z_j\neq 0$ for at least some $i=1,\ldots, 4$ and $j=5,6$.

Assume $\lela X,X\rira=2(x_3x_2-x_1x_4)\neq 0$ then canonical computations show that $k$ vanishes and the (unique) solution to the system is
\begin{eqnarray}\label{eq:solutions}
\xi_1&=&0,\nonumber\\
\xi_2&=&\,\frac{z_5(4\,x_{{4}}x_{{2}}+x_1x_3)+z_6(x_1x_3-4x_4x_2)}{\lela X,X\rira}\nonumber\\
\xi_3&=&-\,\frac{z_5(4\,{x_{{2}}}^{2}+x_1^2)+z_6({x_{{1}}}^{2}-4x_2^2)}{\lela X,X\rira}\nonumber\\
\xi_4&=&\frac{z_5(4\,{x_{{4}}}^{2}+x_3^2)+z_{{6}}({x_{{3}}}^{2}-4 x_4^2)}{{\lela X,X\rira}}.
\end{eqnarray}
This solution is independent of whether $\lela Z,Z\rira$ is zero or not. 

The set \begin{equation}\label{eq:eq51}\mathcal U=\{X+Z:\lela X,X\rira\neq0 \}\end{equation} is an open dense subset of $\mn$ and the map $\xi:\mathcal U\lra \Dera(\mn)$ defined by $\xi(X+Z)=\xi_1 T+\xi_2 H+\xi_3 E+\xi_4 F$ with the coordinates in Equation \prettyref{eq:solutions} is a geodesic graph for $N$.

Set $\mathcal V=\mm\,\backslash\, \mathcal U$ and denote with $\mathcal N$ the null cone of $\mn$, that is $$\mathcal N=\{X+Z\in\mn\,:\, \lela X+Z,X+Z\rira=\lela X,X\rira+\lela Z,Z\rira=0\}.$$ The following lemma describes properties of $\mathcal N$ and $\mathcal V$. 

\begin{lm} \label{lm4} The set $\mathcal V$ is the disjoint union $\mathcal V=\mathcal V_0 \cup\mathcal V_1 \cup\mathcal V_2$, where $\mathcal V_0$, $\mathcal V_1$ and $\mathcal V_2$ are defined below. Every element in $\mathcal V_1$ is a geodesic vector but none of the vectors in $\mathcal V_0$ are geodesic. Moreover the null cone $\mathcal N$ verifies $\mathcal N\cap \mathcal V=\mathcal V_2$.
\end{lm}
\begin{proof} Let $Y=X+Z$ be such that $\lela X,X\rira=0$ and 
$\lela Z,Z\rira\neq 0$. Then the last two rows in \prettyref{eq:matriz} imply $\xi_1=0$ and $k=0$. Also, the vector $(0,\xi_2,\xi_3,\xi_4)$ is a solution of that system if and only if $(\xi_2,\xi_3,\xi_4)$ is a solution of the system $Ax=b$ where
\begin{equation}\label{eq:matrixpeq}\begin{array}{ccc}
A=\left(\begin {array}{ccc}
 x_{{1}}&x_{{3}}&0\\
x_{{2}}&x_{{4}}&0\\
-x_{{3}}&0&x_{{1}}\\
-x_{{4}}&0&x_{{2}}
\end {array} \right) \quad&\mbox{ and }&\quad
b=\left(\begin{array}{c}
-2 ( z_{{5}}-z_{{6}} ) x_{{2}}\\
\frac12 ( z_{{5}}+z_{{6}} ) x_{{1}}\\
-2( z_{{5}}-z_{{6}}) x_{{4}}\\
\frac12( z_{{5}}+z_{{6}} ) x_{{3}}\end{array}\right)\end{array}.\end{equation}

All subdeterminants of order three of $A$ are zero since $\lela X,X\rira=0$. Moreover, $X\neq 0$ implies that some subdeterminants of order two are non-zero, hence $\rk(A)=2$. Frobenius theorem asserts that the system $Ax=b$ admits a solution if only if the rank of $\tilde{A}=(A\,|\,b)$ is 2. 

Every  subdeterminant of order three of $\tilde{A}$ is zero if and only if the following conditions hold simultaneously:
\begin{eqnarray}
z_6(x_3^2-4\,x_4^2) + z_5(4\,x_4^2+ x_3^2)&=&0,\nonumber\\
z_6(x_1^2-4\,x_2^2) + z_5(4\,x_2^2+ x_1^2)&=&0, \label{eq:rango}\\
z_6(x_3x_1-4x_2x_4)+z_5(x_3x_1+4x_2x_4)&=&0. \nonumber
\end{eqnarray}

Define
\begin{eqnarray}
\mathcal V_0&=&\{X+Z\in\mn:\lela X,X\rira=0,\;\lela Z,Z\rira\neq 0,\;\mbox{  \prettyref{eq:rango} does not hold}\},\nonumber\\
\mathcal V_1&=&\{X+Z\in\mn:\lela X,X\rira=0,\;\lela Z,Z\rira\neq 0,\; \mbox{  \prettyref{eq:rango} holds}\}.\nonumber
\end{eqnarray}

For elements in $\mathcal V_0$ there is no solution to the system in \prettyref{eq:matrixpeq} and therefore none of the vectors in  $\mathcal V_0$ is a geodesic vector.

On the contrary, every vector in $\mathcal V_1$ is geodesic. Next, we present the solutions to system \prettyref{eq:matrixpeq} for $ X+Z=\sum_{i=1}^4x_ie_i+\sum_{i=5}^6z_ie_i$ in $\mathcal V_1$.  Given $X+Z\in\mathcal V_1$  one sees that $x_1=0$ if and only if $x_2=0$ and also, $x_3=0$ if and only if $x_4=0$. 
\begin{itemize}
\item If $x_1=0$ then the solutions are $$\xi_2=2(z_5-z_6)\frac{x_4}{x_3}=\frac12 (z_5+z_6)\frac{x_3}{x_4},\quad \xi_3=0, \mbox{ and }\xi_4\in\R.$$
Notice that the equality for $\xi_2$ holds because of the first row in \prettyref{eq:rango}. 
 \item If $x_3=0$ then the solutions are $$\xi_2=-2(z_5-z_6)\frac{x_2}{x_1}=\frac12 (z_5+z_6)\frac{x_1}{x_2},\quad \xi_4=0 \mbox{ and }\xi_3\in\R.$$
 \item If $x_i\neq 0$ for  all $i=1,\ldots,4$ then the solutions are:
$$
\xi_2=2(z_5-z_6)\frac{x_2}{x_1}+\frac{x_2}{x_4}\xi_4,\quad
\xi_3=-4(z_5-z_6)\frac{x_2}{x_3}-\frac{x_1x_2}{x_3x_4}\xi_4, \mbox{ and } \xi_4\in\R.
$$\end{itemize}

Observe that the elements of $\mathcal V_1\cup\mathcal V_0$ have non-zero norm, so $\mathcal N\cap \mathcal V=\mathcal V_2$ where
\begin{eqnarray} \label{eq:v2}
\mathcal V_2&=&\{X+Z\in\mn:\lela X,X\rira=0\mbox{ and }\lela Z,Z\rira=0 \}.\end{eqnarray}
\end{proof}

\begin{teo} \label{teo:teo21} The homogeneous manifold $N=\Iso(N)/\Auto(\mn)$ is an almost g.o. pseudo-Riemannian space which is neither a geodesic orbit space nor a null geodesic orbit space. Moreover, it admits null homogeneous geodesics with non-zero parameter $k$ in the geodesic lemma \prettyref{eq:geolema}.
\end{teo}

\begin{proof}
Since $\mathcal V_0$ in the lemma above is nonempty, $N$ is not a geodesic orbit space with respect to the presentation $\Iso(N)/\Auto(\mn)$. Instead, it is an almost g.o homogeneous space with $\mathcal U$ as in \prettyref{eq:eq51}. 
The set $$\mathcal W=\{ X+Z=\sum_{i=1}^4x_ie_i+\sum_{i=5}^6z_ie_i:\; \lela X,X\rira=0,\, z_5=z_6\neq 0,\, x_1x_2\neq 0\}$$ is a subset of the null cone and $\mathcal W\subseteq \mathcal V_2$ with $\mathcal V_2$ as in \prettyref{eq:v2}. Given $X+Z\in \mathcal W$ define $k=-x_1z_5/x_2$ and set
\begin{equation}
 \xi_1=\frac{1}{2}\, k,\quad \xi_2=\frac32\,\frac{x_1z_5}{x_2}-\,\frac{x_3}{x_1}\,\xi_3,
  \quad \xi_4=3\,\frac{x_3z_5}{x_2}-\,\frac{x_3^2}{x_1^2}\,\xi_3,\quad\xi_3\,\in\,\R.
 \end{equation}
The vector with coordinates $(\xi_1,\xi_2,\xi_3,\xi_4)$ solves the system in \prettyref{eq:matriz}. Thus any element in $\mathcal W$ is an example of a null homogeneous vector with non-zero parameter $k$.

On the other hand, if $Y=X+Z$ is such that $\lela X,X\rira=0$, $z_5=-z_6\neq 0$, $x_1=0$ and $x_2\neq 0$, then it must be $x_3=0$ and there is no solution to the system in \prettyref{eq:matriz}. Since $Y\in\mathcal V_2$, the manifold is not n.g.o.
\end{proof}

\bigskip

To conclude the paper we analyze the properties of this example and the previous results in relation to two conjectures posed by Du\v{s}ek in \cite{DU2}. 
\smallskip

Recall the conjecture in the Introduction. The geodesic lemma allows us to rephrase its statement:

\noindent {\bf Conjecture 1 \cite{DU2}.} 
If an homogeneous space $G/H$ is an almost g.o. space or a g.o. space, then every null homogeneous vector satisfies the geodesic lemma with $k=0$.

The six dimensional nilpotent Lie group modeled on $\R^6$ with multiplication law given in \prettyref{eq:prodN} and endowed with the pseudo-Riemannian left-invariant metric in \prettyref{eq:g} is an almost geodesic orbit space with respect to the presentation $N=\Iso(N)/\Auto(N)$, according to Theorem \ref{teo:teo21}. Moreover, it admits a null homogeneous geodesic with non-vanishing $k$ in the geodesic lemma. Therefore, this constitutes an almost g.o. counterexample to Conjecture 1.

\begin{cor} 
Conjecture 1 is not true for almost g.o. homogeneous spaces.\end{cor}

Our example is not a geodesic orbit space, so the conjecture remains open in that case. We state:

\begin{conj}\label{conjmia}
Let $G/H$ be a pseudo-Riemannian homogeneous g.o. space. For all (null) homogeneous geodesics one has $k=0$ in the geodesic lemma.
\end{conj}

Notice that Theorem \ref{teo:teo7} is a partial proof of Conjecture \ref{conjmia} in the particular case where $G=N$ is a nilpotent Lie group endowed with a left-invariant pseudo-Riemannian metric, acting on itself by left translations so that $H=\{e\}$.

\medskip

\noindent {\bf Conjecture 2. \cite{DU2}}  Let $(G/H,g)$ be a pseudo-Riemannian homogeneous space, and let $\mgg=\mm+\mh$ be a fixed reductive decomposition. Let $\xi$ be the geodesic graph which is nonlinear and unique on an open dense subset $\mathcal U\subset \mm$. 
If the isotropy group $H$ is noncompact, then there is a set $\mathcal V_0\subset \mathcal V=\mm\bs \mathcal U$ where the geodesic graph cannot be defined and hence $(G/H,g)$  is not a geodesic orbit space, but only an almost g.o. space. Further, for any $Y\in\mathcal V$, there is a curve $\gamma(t)$ with the values in  $\mm$ and defined on an interval $[0,\delta)$ such that $\gamma(0)=Y\in\mathcal V$, $\gamma(t)\in\mathcal U$ for $t\in(0,\delta)$ and the limit of some component $\xi_k(\gamma(t))$ of the geodesic graph $\xi$ is infinite for $t\lra 0_+$. 
 \smallskip
 
Our example fits into this statement: the isotropy subgroup $\Auto(\mn)$ is noncompact. In fact, $\Auto_0(\mn)\simeq \GL_0(2)$ as proved in Proposition \ref{pro4}. Also, by Lemma \ref{lm4} the geodesic graph cannot be defined on the set $\mathcal V_0$ consisting of vectors not satisfying Equation \prettyref{eq:rango}. Nevertheless, it is rational on the set $\mathcal U$ in \prettyref{eq:eq51}. For the behavior of the limit of the geodesic graph we have:

\begin{pro} Let $\mathcal V$ be as in Lemma \ref{lm4}. For any vector $Y\in\mathcal V$, there exists a curve $\gamma:[0,\delta)\lra \mn$ such that $\gamma(0)=Y$, $\gamma(t)\in\mathcal U$ for $t>0$ and the limit of $\xi_3(\gamma(t))\lra\infty$ when $t\lra 0_+$.
\end{pro}

\begin{proof}
Let $Y=X+Z$ be in $\mathcal V$ with $X=x_1e_1+x_2e_2+x_3e_3+x_4e_4$ and $Z=z_5e_5+z_6e_6$ as in the previous notations.
Consider the curve $$\gamma(t)=(x_1+t^2,x_2,x_3,x_4+t^4,z_5+t,z_6)$$ and write $\gamma(t)=\gamma_\mv(t)+\gamma_\mz(t)$, where $\gamma_\mv$ and $\gamma_\mz$ are the components on $\mv$ and $\mz$ respectively. Notice that $\lela \gamma_\mv(t),\gamma_\mv(t) \rira =-2\,t^2(t^4+t^2x_1+x_4)$ and $\lela \gamma_\mz(t),\gamma_\mz(t)\rira=t^2+2z_5t$, hence $\lela\gamma_\mv(t),\gamma_\mv(t)\rira\neq 0$ if $0<t<\delta$ for some $\delta>0$. Thus $\gamma(t)\in\mathcal U$ if $0<t<\delta$.

The components of the geodesic graph in \prettyref{eq:solutions} give for $\gamma(t)$
\begin{eqnarray}
\xi_3(\gamma(t))&=&\frac12 \;{\frac { \left( y_{{1}}+{t}^{2} \right) ^{2} \left( y_{{5}}+t+y_{{6}}
 \right) +4\,{y_{{2}}}^{2} \left( y_{{5}}+t-y_{{6}} \right) }{{t}^{6}+{t}^{4}y_{{1}}+t^2y_{{4}}}}.
\end{eqnarray}
One easily sees that that $\lim_{t\lra0^+} \xi_3(\gamma(t))=\infty$.


\end{proof}

\bibliographystyle{plain}
\bibliography{biblio}

\end{document}